\documentclass[12pt,twoside]{amsart}
\usepackage[all]{xy}
\usepackage{graphicx}
\usepackage{amsmath}

\usepackage{longtable}

\title[Toric manifolds whose Chern characters are positive]
{Remarks on toric manifolds whose Chern characters are positive}
\author{Hiroshi Sato and Yusuke Suyama} 
\date{2019/12/6, version 0.12}
\address{Department of Applied Mathematics, Faculty of Sciences, 
Fukuoka University, 8-19-1, Nanakuma, Jonan-ku, Fukuoka 814-0180, Japan}
\email{hirosato@fukuoka-u.ac.jp}
\address{Department of Mathematics, Graduate School of 
Science, Osaka University, Toyonaka, Osaka 560-0043, Japan}
\email{y-suyama@cr.math.sci.osaka-u.ac.jp}

\makeatletter
    
    \@addtoreset{equation}{section}
\makeatother

\newcommand{\Pic}[0]{{\operatorname{Pic}}}

\newcommand{\G}[0]{{\operatorname{G}}}
\newcommand{\ch}[0]{{\operatorname{ch}}}

\newcommand{\relint}[0]{{\operatorname{Relint}}}
\newcommand{\Cone}[0]{{\operatorname{Cone}}}
\newcommand{\Hom}[0]{{\operatorname{Hom}}}
\newcommand{\V}[0]{{\operatorname{V}}}

\newcommand{\pr}[0]{{\operatorname{pr}}}

\newtheorem{thm}{Theorem}[section]
\newtheorem{lem}[thm]{Lemma}
\newtheorem{cor}[thm]{Corollary}
\newtheorem{prop}[thm]{Proposition}

\theoremstyle{definition}
\newtheorem{ex}[thm]{Example}
\newtheorem{defn}[thm]{Definition}

\newtheorem{rem}[thm]{Remark}
\newtheorem*{ack}{Acknowledgments}       

\begin{document}
\bibliographystyle{amsalpha+}

\begin{abstract}
We show that the second Chern character of any projective toric 
manifold of Picard number three is not positive. 
In connection with this result, 
we give various examples of the positivity of higher Chern characters 
of projective toric manifolds. 
\end{abstract}

\thispagestyle{empty}

\maketitle

\tableofcontents
\section{Introduction} %%%%%%%%%%%%%%%%%%%

For a smooth projective variety $X$ over an algebraically closed field
and a positive integer $k$,
the $k$-th Chern character $\mathrm{ch}_k(X)$ is said to be {\em positive}
(resp.\ {\em nef})
if the intersection number $(\mathrm{ch}_k(X)\cdot z)$ is positive (resp.\ non-negative)
for any $z\in\overline{\rm NE}_k(X)\setminus\{0\}$, 
where ${\rm NE}_k(X)$ is the cone of numerical effective $k$-cycles 
on $X$. 
A variety with $\mathrm{ch}_1(X)$ positive is nothing but a Fano variety.
Fano varieties with $\mathrm{ch}_2(X)$ nef are called {\em 2-Fano manifolds}
and they are first studied by de Jong--Starr \cite{ds}
in order to investigate the existence of rational surfaces on a Fano variety.
Fano $d$-folds with $\mathrm{ch}_2(X)$ positive and index $\geq d-2$
are classified in \cite{ac}.

In this paper, we focus on the toric case.
A {\em toric manifold} is a smooth complete toric variety.
For $d\ge 2$,
the $d$-dimensional projective space $\mathbb{P}^d$ is a toric manifold
and one can easily see that the second Chern character $\ch_2(\mathbb{P}^d)$ is positive.
On the other hand,
we do not know toric manifolds whose second Chern character 
is positive other than projective spaces 
at the present moment. Therefore, 
we study projective toric manifolds of Picard number three in order to 
find a new example of such toric manifolds. However, on the contrary, 
we show the following:
\begin{thm}[Theorem \ref{mainthm}]\label{introthm}
The second Chern character of 
any projective toric manifold of Picard number three which does not have 
a Fano contraction is not {\em nef}. 
\end{thm}
Moreover, by combining Theorem \ref{introthm} with 
the known result in \cite{sato3} (see Theorem \ref{extremalsato}), we obtain the following:
\begin{cor}[Corollary \ref{maincor}]
Let $X$ be a smooth projective toric $d$-fold of Picard number at most three. If 
the second Chern character $\ch_2(X)$ of $X$ is {\em positive}, then $X$ is 
isomorphic to the $d$-dimensional projective space $\mathbb{P}^d$. 
\end{cor}

Also, we deal with the positivity of higher Chern characters of projective toric manifolds. 
In contrast to the second Chern character,
we show that there exists a projective toric manifold with $\ch_k(X)$ positive
other than the projective space for each $k \geq 3$. 
We also give some phenomena which do not appear 
in the case of the positivity of the second Chern character. 

\medskip

This paper is organized as follows: In Section \ref{junbi}, we collect 
basic results in the toric geometry. We define the positivity 
of Chern characters for a projective toric manifold. Section \ref{pic3calc} 
is devoted to the calculation of intersection numbers on 
projective toric manifolds of Picard number three which admit no 
Fano contraction. As a result, we can show that the second Chern characters 
of such manifolds are not nef. 
In Section \ref{highercase}, we introduce various phenomena about 
the positivity of higher Chern characters. 

\begin{ack}
The authors would like to thank Professor Osamu Fujino for his supports.
The first author was partially supported by JSPS KAKENHI 
Grant Number JP18K03262. 
The second author was partially supported 
by JSPS KAKENHI 
Grant Number JP18J00022. 
\end{ack}

\section{Preliminaries}\label{junbi} %%%%%%%%%%%%%%%%%

In this section, we introduce the fundamental notation and concepts in the toric geometry. 
For details, please see 
\cite{cls}, \cite{fulton} and \cite{oda} for the toric geometry. Also see 
\cite{fujino-sato}, \cite{matsuki} and \cite{reid} for 
the toric Mori theory. 
We will work over an algebraically closed field $K=\overline{K}$. 

Let $N:=\mathbb{Z}^d$, $M:=\Hom_{\mathbb{Z}}(N,\mathbb{Z})$, 
$N_{\mathbb{R}}:=N\otimes_{\mathbb{Z}}\mathbb{R}$ and 
$M_{\mathbb{R}}:=M\otimes_{\mathbb{Z}}\mathbb{R}$. 
$\Cone(S)$ 
stands for the cone generated by $S\subset N_{\mathbb{R}}$. 
For a fan $\Sigma$ in $N$, we denote by $X=X_\Sigma$ 
the toric $d$-fold associated to $\Sigma$. 
We denote by $T=T_N$ the algebraic torus for $N$.  
For a smooth complete fan $\Sigma$, put 
\[
\G(\Sigma):=\left\{\mbox{the primitive generators of 
$1$-dimensional cones in $\Sigma$}\right\}\subset N.
\]
For a cone $\sigma\in\Sigma$, 
we put $\G(\sigma):=\sigma\cap\G(\Sigma)$, that is, the set of 
primitive generators of 
$\sigma$. 
For any $x\in\G(\Sigma)$, there is the corresponding torus invariant 
divisor $D_x$ on $X$.
It is well-known that there exists the following isomorphism
of abelian groups:
\[
{\mathrm A}_1(X)\cong 
\left\{(a_x)_{x\in\G(\Sigma)}\in \mathbb{Z}^{\G(\Sigma)}\,\left|\,
\sum_{x\in\G(\Sigma)}a_xx=0\right.\right\},
\]
where ${\mathrm A}_1(X)$ is the group of numerical $1$-cycles on $X$. 
Thus, we can regard a linear relation among elements in $\G(\Sigma)$ as 
a numerical $1$-cycle on $X$. 
For a $(d-1)$-dimensional cone $\tau\in\Sigma$, there exists a linear relation 
\[
y_1+y_2+a_1x_1+\cdots+a_{d-1}x_{d-1}=0,
\]
where $\G(\tau)=\{x_1,\ldots,x_{d-1}\}$, $a_1,\ldots,a_{d-1}\in\mathbb{Z}$ and 
for $\{y_1,y_2\}\subset\G(\Sigma)$, 
$\G(\tau)\cup\{y_1\}$ and $\G(\tau)\cup\{y_2\}$ generate 
maximal cones in $\Sigma$. 
This linear relation corresponds to the torus invariant curve $C_{\tau}$ 
associated to $\tau$. We remark that 
\[
2+a_1+\cdots+a_{d-1}=(-K_X\cdot C_{\tau}).
\]
The following relation is convenient to describe a complete smooth fan:
\begin{defn}[{\cite[Definitions 2.6, 2.7 and 2.8]{bat1}}]\label{pcpc}
Let $X=X_\Sigma$ be a toric manifold. 
We call a nonempty subset $P\subset \G(\Sigma)$ a {\em primitive collection} in 
$\Sigma$ if $P$ does not generate a cone in $\Sigma$, while 
any proper subset of $P$ generates a cone in $\Sigma$. 

For a primitive collection $P=\{x_1,\ldots,x_r\}$, there exists the unique cone 
$\sigma(P)\in\Sigma$ such that $x_1+\cdots+x_r$ is contained in the 
relative interior of $\sigma(P)$. Put $\G(\sigma)=
\{y_1,\ldots,y_s\}$. Then, we have a linear relation
\[
x_1+\cdots+x_r=a_1y_1+\cdots+a_sy_s\ (a_1,\ldots,a_s\in\mathbb{Z}_{>0}).
\]
We call this relation the {\em primitive relation} of $P$. 
Thus, by the above argument, 
we obtain the numerical $1$-cycle $r(P)\in{\mathrm A}_1(X)$ for any 
primitive collection $P\subset\G(\Sigma)$. 
We remark that 
\[
r-(a_1+\cdots+a_s)=\left(-K_X\cdot r(P)
\right).
\]
\end{defn}

We can characterize toric Fano manifolds using the notion of primitive relations 
as follows:

\begin{thm}[{\cite[Proposition 2.3.6]{bat2}}]\label{Fano}
Let $X=X_\Sigma$ be a toric manifold.
Then $X$ is Fano if and only if $\left(-K_X\cdot r(P)\right)>0$
for any primitive collection $P$ in $\Sigma$.
\end{thm}

In this paper, we study the positivity of the $k$-th Chern character $\ch_k(X)$ for 
a smooth projective toric $d$-fold $X=X_\Sigma$. 
Let $D_1,\ldots,D_n$ be the torus invariant prime divisors on $X$. Then, 
$\ch_k(X)$ is described as 
\[
\ch_k(X)=\frac{1}{k!}\left(D_1^k+\cdots+D_n^k\right). 
\] 
The following proposition is fundamental for our theory. 
The proof is completely similar to the one for
\cite[Proposition 2.26]{oda} in the Japanese version. 
We describe a sketch of the proof for the reader's convenience. 
\begin{prop}\label{torusinv}
Let $X=X_\Sigma$ be a toric manifold of $\dim X=d$ and $1\le k\le d$. For any $k$-dimensional 
irreducible closed subvariety $Y\subset X$, there exist torus invariant 
$k$-dimensional irreducible closed subvarieties $Y_1,\ldots,Y_l$ such that 
$Y$ is numerically equivalent to 
\[
a_1Y_1+\cdots+a_lY_l
\]
for positive integers $a_1,\ldots,a_l$. 
\end{prop}
\begin{proof}[Sketch of the proof]
Let $\V(\sigma)$ be the smallest torus invariant irreducible closed subvariety 
associated to a cone $\sigma\in\Sigma$ 
which contains $Y$. Obviously, $\dim \sigma\le d-k$. So, there exists a 
$(d-k)$-dimensional cone $\tau\in\Sigma$ such that 
$\sigma$ is a face of $\tau$.  
For $n\in\left(\relint (\tau)\right)\cap N$, where $\relint (\tau)$ 
stands for the relative interior of $\tau$, 
let $\gamma_n:K^{\times}\to T$ be the corresponding one-parameter subgroup. 
For a morphism 
\[
\begin{array}{cccc}
\Phi: & \V(\sigma)\times K^{\times} & \to & \V(\sigma)\times K^{\times}, \\
&\rotatebox{90}{$\in$} & & \rotatebox{90}{$\in$} \\
&(u,\lambda) & \mapsto & (\gamma_n(\lambda)u,\lambda)
\end{array}
\]
let $Z:=\overline{\Phi\left(Y\times K^{\times}\right)}$ be the closure of 
$\Phi\left(Y\times K^{\times}\right)$ in 
$\V(\sigma)\times K$ and $\pr_2:Z\to K$ the second projection. Then, 
there exist $(d-k)$-dimensional cones $\tau_1,\ldots,\tau_l$ which contain $\sigma$ 
as a face such that 
\[
\pr_2^{-1}(1)=Y,\mbox{ while }\pr_2^{-1}(0)=\V(\tau_1)\cup\cdots\cup\V(\tau_l)
\]
as sets, where $\V(\tau_1),\ldots,\V(\tau_l)$ are the 
torus invariant irreducible closed subvarieties 
associated to $\tau_1,\ldots,\tau_l$, respectively 
(see \cite[Proposition 3.2.2]{cls} and \cite[Proposition 1.6 (v)]{oda}). 
Therefore, $Y$ is rationally equivalent to 
\[
a_1\V(\tau_1)+\cdots+a_l\V(\tau_l)
\]
for positive integers $a_1,\ldots,a_l$. 
In particular, they are numerically equivalent. 
\end{proof}
Proposition \ref{torusinv} tells us that for a toric manifold $X$, 
we may define the positivity (resp. non-negativity) of $\ch_k(X)$ as follows:
\begin{defn}\label{positivedef}
Let $X$ be a toric manifold of $\dim X=d$. For $1\le k\le d$, we say
$\ch_k(X)$ is {\em positive} (resp. {\em nef}) if $(\ch_k(X)\cdot Y)>0$ (resp. $\ge 0$) 
for any 
$k$-dimensional closed torus invariant subvariety $Y\subset X$. 
When $\ch_k(X)$ is positive (resp. nef), we say $X$ is $\ch_k$-{\em positive} (resp. $\ch_k$-{\em nef}). 
\end{defn}

\begin{rem}
When $k=1$, we have $\ch_1(X)=\mathrm{c}_1(X)=-K_X$. So, a $\ch_1$-positive (resp. $\ch_1$-nef) 
toric manifold is nothing but a 
toric Fano (resp. weak Fano) manifold. 
\end{rem}

It is not difficult to see that the $d$-dimensional projective space $\mathbb{P}^d$ is 
$\ch_k$-positive for $1\le k\le d$. Moreover, for $k=2$, the following holds:
\begin{thm}[{\cite[Theorem 1.3]{sato3}}]\label{extremalsato}
Let $X=X_\Sigma$ be a smooth projective toric $d$-fold. 
If $X=X_\Sigma$ is $\ch_2$-positive and 
$X$ has a Fano contraction, then $X\cong\mathbb{P}^d$. 
In particular, if $\rho(X)=2$, then $X$ is not $\ch_2$-positive. 
\end{thm}

\section{Toric manifolds of Picard number three}\label{pic3calc} %%%%%%%%%%%%%
%%%%%%%%%%%%%%%%%%%%%%%%%%%%%%%%%%
%%%%%%%%%%%%%%%%%%%%%%%%%%

In this section, we determine whether a projective toric manifold $X$ of 
$\rho(X)=3$ is $\ch_2$-positive or not. By Theorem \ref{extremalsato}, 
we may assume $X$ has no Fano contraction. 
We use the notation of \cite{bat1} throughout this section:
 
\begin{thm}[{\cite[Theorem 6.6]{bat1}}]\label{batpic3}
Let $X=X_\Sigma$ be a projective toric manifold of Picard number three. 
If $X$ has no Fano contraction, then the fan $\Sigma$ is 
explicitly described as follows:
Let 
\[
\G(\Sigma)=\{
v_1,\ldots,v_{p_0},
y_1,\ldots,y_{p_1},
z_1,\ldots,z_{p_2},
t_1,\ldots,t_{p_3},
u_1,\ldots,u_{p_4}
\}, 
\]
where $p_0,p_1,p_2,p_3,p_4$ are positive integers and 
$p_0+p_1+p_2+p_3+p_4-3=d:=\dim X$. Then, the primitive relations of $\Sigma$ are 
\[
v_1+\cdots+v_{p_0}+y_{1}+\cdots+y_{p_1}=c_2z_2+\cdots+c_{p_2}z_{p_2}+(b_1+1)t_1+\cdots+(b_{p_3}+1)t_{p_3},
\]
\[
y_1+\cdots+y_{p_1}+z_1+\cdots+z_{p_2}=u_1+\cdots+u_{p_4},
\]
\[
z_1+\cdots+z_{p_2}+t_1+\cdots+t_{p_3}=0,
\]
\[
t_1+\cdots+t_{p_3}+u_1+\cdots+u_{p_4}=y_{1}+\cdots+y_{p_1}\mbox{ and}
\]
\[
u_1+\cdots+u_{p_4}+v_1+\cdots+v_{p_0}=c_2z_2+\cdots+c_{p_2}z_{p_2}+b_1t_1+\cdots+b_{p_3}t_{p_3}
\]
for $b_1,\ldots,b_{p_3},c_2,\ldots,c_{p_2}\in\mathbb{Z}_{\ge 0}$. 
\end{thm}

We may assume $c_2=\min\{c_2,\ldots,c_{p_2}\}$ and 
$b_1=\min\{b_1,\ldots,b_{p_3}\}$.

Let 
\[
V_1,\ldots,V_{p_0},
Y_1,\ldots,Y_{p_1},
Z_1,\ldots,Z_{p_2},
T_1,\ldots,T_{p_3},
U_1,\ldots,U_{p_4}
\]
be the torus invariant divisors in $\Pic(X)$ associated to 
\[
v_1,\ldots,v_{p_0},
y_1,\ldots,y_{p_1},
z_1,\ldots,z_{p_2},
t_1,\ldots,t_{p_3},
u_1,\ldots,u_{p_4},
\] 
respectively. 
By calculating the rational functions for the dual basis of 
$\G(\Sigma)\setminus\{v_1,z_1,u_1\}$, we have the relations
\[
V_2-V_1=0,\ldots,V_{p_0}-V_1=0,
\]
\[
Y_1+U_1-V_1=0,\ldots,Y_{p_1}+U_1-V_1=0,
\]
\[
Z_2-Z_1+c_2V_1=0,\ldots,Z_{p_2}-Z_1+c_{p_2}V_1=0,
\]
\[
T_1-Z_1-U_1+(b_1+1)V_1=0,\ldots,T_{p_3}-Z_1-U_1+(b_{p_3}+1)V_1=0,
\]
\[
U_2-U_1=0,\ldots,U_{p_4}-U_1=0
\]
in $\Pic(X)$. In particular, we obtain
\[
V_1=\cdots=V_{p_0},\ Y_1=\cdots =Y_{p_1},\ 
U_1=\cdots=U_{p_4},\ T_i=T_1+(b_1-b_i)V_1\ (2\le i\le p_3).
\]
On the other hand, the primitive collections of $\Sigma$ tell us that 
\[
V_1\cdots V_{p_0}\cdot Y_1\cdots Y_{p_1}=0,\ 
Y_1\cdots Y_{p_1}\cdot Z_1\cdots Z_{p_2}=0,\ 
Z_1\cdots Z_{p_2}\cdot T_1\cdots T_{p_3}=0,
\]
\[
T_1\cdots T_{p_3}\cdot U_1\cdots U_{p_4}=0\mbox{ and }
U_1\cdots U_{p_4}\cdot V_1\cdots V_{p_0}=0.
\]
Finally, we should remark that 
\[
2\ch_2(X)=V^2_1+\cdots+V^2_{p_0}+
Y^2_1+\cdots+Y^2_{p_1}+
Z^2_1+\cdots+Z^2_{p_2}+
T^2_1+\cdots+T^2_{p_3}+
U^2_1+\cdots+U^2_{p_4}
\]
for our calculation below. 
%Suppose that $\ch(X)$ is nef. 

\medskip

%Let $S$ be the torus invariant surface 
%associated to a $(d-2)$-dimensional cone 
%$\tau$. 
%We may assume 
%$b_1\le\cdots\le b_{p_3}$ and 
%$c_2\le\cdots\le c_{p_2}$. 

\medskip

First of all, we calculate $(\ch_2(X)\cdot S_1)$ for 
the torus invariant subsurface $S_1$ associated to 
the $(d-2)$-dimensional cone $\tau\in\Sigma$ such that 
\underline{$\G(\tau)=\G(\Sigma)\setminus
\{v_{1},y_{1},z_{1},t_{1},u_{1}\}$}. In this case, 
$\rho(S_1)=3$. 
Since 
\[
S_1=V_2\cdots V_{p_0}\cdot Y_2\cdots Y_{p_1}\cdot Z_2\cdots Z_{p_2}
\cdot T_2\cdots T_{p_3}\cdot U_2\cdots V_{p_4},
\]
we have 
\[
V_1\cdot Y_1\cdot S_1=Y_1\cdot Z_1\cdot S_1=Z_1\cdot 
T_1\cdot S_1=T_1\cdot U_1\cdot S_1=U_1\cdot V_1\cdot S_1=0.
\]
We can calculate the intersection numbers as follows:
\[
(V_1^2\cdot S_1)=\left(V_1\cdot (Y_1+U_1)\cdot S_1\right)=0,
\]
\[
(Y_1^2\cdot S_1)=\left(Y_1\cdot (V_1-U_1)\cdot S_1\right)=-1,
\]
\[
(U_1^2\cdot S_1)=\left(U_1\cdot (V_1-Y_1)\cdot S_1\right)=-1,
\]
\[
(Z_1^2\cdot S_1)=\left(Z_1\cdot (T_1-U_1+(b_1+1)V_1)\cdot S_1\right)=-1+b_1+1=b_1,
\]
\[
(Z_2^2\cdot S_1)=\left( (Z_1-c_2V_1)\cdot (Z_1-c_2V_1)\cdot S_1\right)
=(Z_1^2\cdot S_1)-2c_2(Z_1\cdot V_1\cdot S_1)+c_2^2(V_1^2\cdot S_1)=b_1-2c_2,
\]
\[
\vdots
\]
\[
(Z_{p_2}^2\cdot S_1)=b_1-2c_{p_2},
\]
\[
(T_1^2\cdot S_1)=\left( T_1\cdot (Z_1+U_1-(b_1+1)V_1)\cdot S_1\right)=-(b_1+1),
\]
\[
(T_2^2\cdot S_1)=\left( (T_1+(b_1-b_2)V_1)\cdot (Z_1+U_1-(b_2+1)V_1)\cdot S_1\right)
\]
\[
=(T_1\cdot(-(b_2+1)V_1)\cdot S_1)+((b_1-b_2)V_1\cdot Z_1\cdot S_1)
+((b_1-b_2)V_1\cdot (-(b_2+1)V_1)\cdot S_1)
\]
\[
=-(b_2+1)+(b_1-b_2)=b_1-2b_2-1,
\]
\[
\vdots
\]
\[
(T_{p_3}^2\cdot S_1)=-(b_{p_3}+1)+(b_1-b_{p_3})=b_1-2b_{p_3}-1.
\]
Therefore, 
\begin{align*}
2(\ch_2(X)\cdot S_1) & =-p_1-p_4+b_1p_2
-2(c_2+\cdots+c_{p_2}) \\
& -(b_1+1)+(b_1-b_2-(b_2+1))+\cdots+
(b_1-b_{p_3}-(b_{p_3}+1)).
\end{align*}
So, we have the following:
\begin{lem}\label{p2p3}
If $\ch_2(X)$ is nef, then 
$b_1\ge 1$ and $p_2>p_3$. In particular, $p_2\ge 2$. 
\end{lem}
\begin{proof}
By the above calculation, 
\[
2(\ch_2(X)\cdot S_1)
\le -p_1-p_4+b_1p_2-(b_1+1)p_3=-p_1-p_3-p_4+b_1(p_2-p_3).
\]
Thus, $b_1$ and $p_2-p_3$ have to be positive. 
\end{proof}

Next, we treat the case where the Picard number of the torus invariant subsurface 
$S=S_\tau\subset X_{\Sigma}$ associated to a $(d-2)$-dimensional cone $\tau\in\Sigma$ 
is two. 
In this case, we can apply the following result: 
\begin{prop}[{\cite[Theorem 3.5 and Proposition 4.4]{sato2}}]\label{polyinter}
Let $X=X_{\Sigma}$ be a projective toric manifold and $S=S_{\tau}$ 
the torus invariant subsurface associated to a $(d-2)$-dimensional cone 
$\tau\in\Sigma$ in $X$ such that $\rho(S)=2$. 
In this case, $S$ is isomorphic to the Hirzebruch surface $F_\alpha$ of degree $\alpha\ge 0$. 
Let $\G(\tau)=\{x_1,\ldots,x_{d-2}\}$. We have exactly four maximal cones 
\[
\Cone(\G(\tau)\cup\{w_1,w_2\}),\ \Cone(\G(\tau)\cup\{w_2,w_3\}), 
\]
\[
\Cone(\G(\tau)\cup\{w_3,w_4\})\mbox{ and }\Cone(\G(\tau)\cup\{w_4,w_1\})
\]
which contain $\tau$ 
for $w_1,w_2,w_3,w_4\in\G(\Sigma)$ such that 
\[
w_1+w_3+a_1x_1+\cdots +a_{d-2}x_{d-2}=0\mbox{ and }
w_2+w_4-\alpha w_1+e_1x_1+\cdots +e_{d-2}x_{d-2}=0.
\]
Then, 
\[
(\ch_2(X)\cdot S)=\frac{1}{2}\left(\alpha(2+a_1^2+\cdots+a_{d-2}^2)
+2(-\alpha+a_1e_1+\cdots +a_{d-2}e_{d-2})\right).
\]
\end{prop}

We calculate the intersection number $(\ch_2(X)\cdot S)$ for two cases: 
(1) 
$\G(\tau)=\G(\Sigma)\setminus\{v_1,z_1,z_2,t_1,t_2\}$ and 
(2) 
$\G(\tau)=\G(\Sigma)\setminus\{v_1,y_1,z_1,z_2,u_1\}$. 

\begin{enumerate}
\item \underline{$\G(\tau)=\G(\Sigma)\setminus\{v_1,z_1,z_2,t_1,t_2\}$}.
We have two linear relations
\[
z_1+z_2+z_3+\cdots+z_{p_2}+y_1+\cdots+y_{p_1}-(u_1+\cdots+u_{p_4})=0\mbox{ and}
\]
\[
t_1+t_2+t_3+\cdots+t_{p_3}+u_1+\cdots+u_{p_4}-(y_1+\cdots+y_{p_1})=0
\]
for $\Cone(\G(\tau)\cup\{t_1\})$ and $\Cone(\G(\tau)\cup\{z_1\})$, respectively. 
One can easily see that $S\cong \mathbb{P}^1\times\mathbb{P}^1$, and we obtain 
\[
2(\ch_2(X)\cdot S)=2(-p_1-p_4)<0 
\]
by Proposition \ref{polyinter}. 
So, we have the following by Lemma \ref{p2p3}:
\begin{lem}\label{p3is1}
If $\ch_2(X)$ is nef, then $p_3=1$. 
\end{lem}
\begin{proof}
If $p_2\ge 2$ and $p_3\ge 2$, then $(\ch_2(X)\cdot S)<0$. 
Lemma \ref{p2p3} says that $p_2\ge 2$. So, we have $p_3=1$. 
\end{proof}
Therefore, we may assume $p_3=1$ in the following calculation.

\item 
\underline{$\G(\tau)=\G(\Sigma)\setminus\{v_1,y_1,z_1,z_2,u_1\}$.} 
First, we remark that this subsurface always exists by Lemma \ref{p2p3}.  
We have two linear relations
\[
z_1+z_2+\cdots+z_{p_2}+t_1=0\mbox{ and}
\]
\[
v_1+y_1+v_2+\cdots+v_{p_0}+y_{2}+\cdots+y_{p_1}-(c_2z_2+\cdots+c_{p_2}z_{p_2}+(b_1+1)t_1)=0
\]
for $\Cone(\G(\tau)\cup\{v_1\})$ and $\Cone(\G(\tau)\cup\{z_2\})$, respectively. 
One can easily see that $S\cong F_{c_2}$, and we obtain 
\[
2(\ch_2(X)\cdot S)=c_2(p_2+1)+2(-(c_2+\cdots+c_{p_2}+b_1+1))
\]
by Proposition \ref{polyinter}. 
If $p_2\ge 3$, we have 
\[
2(\ch_2(X)\cdot S)=c_2+(c_2-c_2)+\cdots+(c_2-c_{p_2})+c_2-(c_2+\cdots+c_{p_2})-2(b_1+1)<0.
\]
Thus, we obtain the following lemma (see Lemma \ref{p2p3}):
\begin{lem}\label{p2p22}
If $\ch_2(X)$ is nef, then $p_2=2$ and $c_2-2(b_1+1)\ge 0$. 
\end{lem}

\end{enumerate}

Thus, we conclude the following:

\begin{thm}\label{mainthm}
Let $X$ be a projective toric manifold. If $\rho(X)=3$ and 
$X$ has no Fano contraction, 
then $\ch_2(X)$ is {\em not} nef.
\end{thm}
\begin{proof}
Suppose that $\ch_2(X)$ is nef. Then, we may assume that $p_2=2$, $p_3=1$, $b_1\ge 1$ 
and $c_2-2(b_1+1)\ge 0$ by Lemmas \ref{p2p3}, \ref{p3is1} and \ref{p2p22}.
However, the last inequality 
obviously contradicts the another 
inequality  
\[
2(\ch_2(X)\cdot S_1)=-p_1-p_4+2b_1-2c_2-(b_1+1)\ge 0
\]
\[
\Longleftrightarrow\ b_1\ge2c_2+p_1+p_4+1.
\]
Thus, $\ch_2(X)$ is not nef.
\end{proof}

By combining Theorems \ref{extremalsato} and \ref{mainthm}, we have the following:

\begin{cor}\label{maincor}
If $X=X_\Sigma$ is a $\ch_2$-positive smooth projective toric $d$-fold 
and $\rho(X)\le 3$, then 
$X\cong\mathbb{P}^d$. 
\end{cor}

\section{Positivities of higher Chern characters}\label{highercase} %%%%%%%%%%%%%%%%%%%
%%%%%%%%%%%%%%%%%%%%%%%%%%%%%%%%%%%%%%
%%%%%%%%%%%%%%%%%%%%%%%%%%%%%

In this section we study toric manifolds $X$ with $\mathrm{ch}_k(X)$ nef for $k \geq 3$.
In the case $d=k$, it is not difficult to construct
toric $d$-folds $X$ with $\mathrm{ch}_d(X)$ positive.

\begin{ex}
Let $d \geq 3$ and $a \geq 1$.
Let $X=X_\Sigma$ be a toric manifold of $\rho(X)=2$
such that $\G(\Sigma)=\{x_1, x_2, x_3, y_1, \ldots, y_{d-1}\}$
and the primitive relations of $\Sigma$ are
\begin{equation*}
x_1+x_2+x_3=ay_1 \mbox{ and } y_1+\cdots+y_{d-1}=0.
\end{equation*}
Then $X$ is isomorphic to the $\mathbb{P}^{d-2}$-bundle
\[
\mathbb{P}_{\mathbb{P}^2}\left(\mathcal{O}(a) \oplus \mathcal{O}^{\oplus d-2}\right)
\]
over $\mathbb{P}^2$.
We denote by
$D_1, D_2, D_3, E_1, \ldots, E_{d-1}$
the torus invariant divisors in $\Pic(X)$
associated to $x_1,x_2,x_3,y_1,\ldots,y_{d-1}$, respectively.
Then we have the relations
\begin{equation*}
D_1=D_2=D_3 \mbox{ and } aD_3+E_1=E_2=\cdots=E_{d-1}
\end{equation*}
in $\Pic(X)$.
We can calculate the intersection numbers as follows:
\begin{align*}
D_1^d&=D_2^d=D_3^d=0,\\
E_1^d&=E_1^2(E_2-aD_3)^{d-2}\\
&=E_1^2\left(E_2^{d-2}-a(d-2)D_3 \cdot E_2^{d-3}
+a^2\frac{(d-2)(d-3)}{2}D_3^2 \cdot E_2^{d-4}\right)\\
&=E_1^2 \cdot E_2 \cdots E_{d-1}
-a(d-2) D_3 \cdot E_1^2 \cdot E_2 \cdots E_{d-2} \\
&+a^2\frac{(d-2)(d-3)}{2} D_2 \cdot D_3 \cdot E_1^2 \cdot E_2 \cdots E_{d-3}\\
&=0-a(d-2)\cdot(-a)+a^2\frac{(d-2)(d-3)}{2}\\
&=a^2(d-2)+a^2\frac{(d-2)(d-3)}{2}>0,\\
E_2^d&=\cdots=E_{d-1}^d=(aD_3+E_1) \cdot E_2 \cdot E_3 \cdots E_{d-1}^2
=aD_3 \cdot E_2 \cdot E_3 \cdots E_{d-1}^2=a \cdot a=a^2>0.
\end{align*}
Thus, $\mathrm{ch}_d(X)$ is positive.
Note that $X$ is Fano if $a \leq 2$
(this follows from Theorem \ref{Fano}).
\end{ex}

Next we consider the case $d>k$.
For odd $k$, there is a toric $d$-fold $X$ with $\mathrm{ch}_k(X)$ positive.
However, our example does not provide
an example of a toric manifold with $\mathrm{ch}_k(X)$ positive for even $k$.

\begin{prop}\label{p1bdle}
Let $d>k \geq 3$ and $a \geq 1$ with $d-a^k \geq 1$.
Let $X=X_\Sigma$ be a toric manifold of $\rho(X)=2$
such that $\G(\Sigma)=\{x_1, \ldots, x_d, y_1, y_2\}$
and the primitive relations of $\Sigma$ are
\begin{equation*}
x_1+\cdots+x_d=ay_1 \mbox{ and } y_1+y_2=0.
\end{equation*}
Then the following hold:
\begin{enumerate}
\item If $k$ is odd, then $\mathrm{ch}_k(X)$ is positive.
\item If $k$ is even, then $\mathrm{ch}_k(X)$ is nef, but not positive.
\end{enumerate}
\end{prop}

\begin{rem}
In Proposition \ref{p1bdle},
$X$ is isomorphic to the $\mathbb{P}^1$-bundle
$\mathbb{P}_{\mathbb{P}^{d-1}}(\mathcal{O}(a) \oplus \mathcal{O})$
over $\mathbb{P}^{d-1}$, and it is Fano by Theorem \ref{Fano}.
\end{rem}

\begin{proof}[Proof of Proposition \ref{p1bdle}]
We denote by
$D_1, \ldots, D_d, E_1, E_2$
the torus invariant divisors in $\Pic(X)$
associated to $x_1, \ldots, x_d, y_1, y_2$, respectively.
Then we have the relations
\begin{equation*}
D_1=\cdots=D_d \mbox{ and } aD_d+E_1=E_2
\end{equation*}
in $\Pic(X)$.
Put $\tau_1$ (resp.\ $\tau_2$) as the $(d-k)$-dimensional cone in $\Sigma$
whose generators are $x_1, \dots, x_{d-k}$ (resp.\ $x_1, \dots, x_{d-k-1}, y_1$).
We denote by $V_i$ the $k$-dimensional torus invariant closed subvariety
of $X$ associated to $\tau_i$ for $i=1, 2$.
It follows from Proposition \ref{torusinv} and the relations in $\Pic(X)$
that $\mathrm{ch}_k(X)$ is positive (resp.\ nef)
if and only if $(\mathrm{ch}_k(X) \cdot V_i)$ is positive (resp. non-negative)
for each $i=1, 2$.
Since
\begin{align*}
D_1^k \cdot D_1 \cdots D_{d-k}&=D_2^k \cdot D_1 \cdots D_{d-k}
=\cdots=D_d^k \cdot D_1 \cdots D_{d-k}=0, \\
E_1^k \cdot D_1 \cdots D_{d-k}&=D_1 \cdots D_{d-k} \cdot (E_2-aD_d)^{k-2} \cdot E_1^2
=D_1 \cdots D_{d-k} \cdot (-aD_d)^{k-2} \cdot E_1^2 \\
&=(-a)^{k-2} D_1 \cdots D_{d-2} \cdot E_1^2=(-a)^{k-1}, \\
E_2^k \cdot D_1 \cdots D_{d-k}&=D_1 \cdots D_{d-k} \cdot (aD_d+E_1)^{k-2} \cdot E_2^2
=D_1 \cdots D_{d-k} \cdot (aD_d)^{k-2} \cdot E_2^2 \\
&=a^{k-2} D_1 \cdots D_{d-2} \cdot E_2^2=a^{k-1},
\end{align*}
we have
\begin{equation*}
(\mathrm{ch}_k(X) \cdot V_1)=\left\{\begin{array}{ll}
2a^{k-1} & (k \mbox{ is odd}), \\
0 & (k \mbox{ is even}). \end{array}\right.
\end{equation*}
Since
\begin{align*}
D_1^k \cdot D_1 \cdots D_{d-k-1} \cdot E_1&=D_2^k \cdot D_1 \cdots D_{d-k-1} \cdot E_1
=\cdots=D_d^k \cdot D_1 \cdots D_{d-k-1} \cdot E_1=1, \\
E_1^k \cdot D_1 \cdots D_{d-k-1} \cdot E_1
&=D_1 \cdots D_{d-k-1} \cdot (E_2-aD_d)^{k-1} \cdot E_1^2 \\
&=(-a)^{k-1} D_1 \cdots D_{d-2} \cdot E_1^2=(-a)^k, \\
E_2^k \cdot D_1 \cdots D_{d-k-1} \cdot E_1&=0,
\end{align*}
we have
\begin{equation*}
(\mathrm{ch}_k(X) \cdot V_2)=\left\{\begin{array}{ll}
d-a^k & (k \mbox{ is odd}), \\
d+a^k & (k \mbox{ is even}). \end{array}\right.
\end{equation*}
Thus, $\mathrm{ch}_k(X)$ is nef for any $k \geq 3$,
and it is positive if and only if $k$ is odd.
\end{proof}

At present, the only known examples of toric $d$-folds
with $\mathrm{ch}_4(X)$ positive for $d \geq 5$ are projective spaces.
Finally, we prove the following theorem:

\begin{thm}
Any smooth complete toric $d$-fold $X$ of $\rho(X)=2$ with $d \geq 5$
is not $\mathrm{ch}_4$-positive.
\end{thm}

\begin{proof}
Any toric manifold of Picard number two
is a projective space bundle over a projective space.
Let $s-1$ be the dimension of a fiber.
If $s=2$, then $\mathrm{ch}_4(X)$ is not positive by Proposition \ref{p1bdle}.
Thus, it suffices to prove the assertion in the case $s \geq 3$.

(1) {\it The case where $s=3$}.
Let $X=X_\Sigma$ be a toric manifold of $\rho(X)=2$
such that $\G(\Sigma)=\{x_1, \ldots, x_{d-1}, y_1, y_2, y_3\}$
and the primitive relations of $\Sigma$ are
\begin{equation*}
x_1+\cdots+x_{d-1}=a_1y_1+a_2y_2 \mbox{ and } y_1+y_2+y_3=0,
\end{equation*}
where $a_1 \geq a_2 \geq 0$.
Then $X$ is the $\mathbb{P}^2$-bundle
$\mathbb{P}_{\mathbb{P}^{d-2}}(\mathcal{O}(a_1) \oplus \mathcal{O}(a_2) \oplus \mathcal{O})$
over $\mathbb{P}^{d-2}$.
We denote by
$D_1, \ldots, D_{d-1}, E_1, E_2, E_3$
the torus invariant divisors in $\Pic(X)$
associated to $x_1, \ldots, x_{d-1}, y_1, y_2, y_3$, respectively.
Then we have the relations
\begin{equation*}
D_1=\cdots=D_{d-1} \mbox{ and } a_iD_{d-1}+E_i-E_3=0
\end{equation*}
for $i=1, 2$ in $\Pic(X)$.
Put $\tau$ as the $(d-4)$-dimensional cone in $\Sigma$
whose generators are $x_1, \ldots, x_{d-5}, y_1$.
We denote by $V$ the $4$-dimensional torus invariant closed subvariety
of $X$ associated to $\tau$.
We show $(\mathrm{ch}_4(X) \cdot V) \leq 0$.
Since
$D_i^4 \cdot D_1 \cdots D_{d-5} \cdot E_1=0$ for every $i=1, \ldots, {d-1}$,
$E_2^4 \cdot D_1 \cdots D_{d-5} \cdot E_1=-a_2^3$
and $E_3^4 \cdot D_1 \cdots D_{d-5} \cdot E_1=a_2^3$,
it suffices to show $E_1^4 \cdot D_1 \cdots D_{d-5} \cdot E_1 \leq 0$.
We calculate the intersection number as follows:
\begin{align*}
&E_1^4 \cdot D_1 \cdots D_{d-5} \cdot E_1 \\
&=D_1 \cdots D_{d-5} \cdot E_1^3 \cdot (E_2-(a_1-a_2)D_{d-1}) \cdot (E_3-a_1D_{d-1}) \\
&=D_1 \cdots D_{d-5} \cdot E_1^3
\cdot (-a_1D_{d-1} \cdot E_2-(a_1-a_2)D_{d-1} \cdot E_3 +a_1(a_1-a_2)D_{d-1}^2) \\
&=D_1 \cdots D_{d-5} \cdot E_1^3
\cdot (-a_1D_{d-1} \cdot E_2-(a_1-a_2)D_{d-1} \cdot (a_2D_{d-1}+E_2)+a_1(a_1-a_2)D_{d-1}^2) \\
&=D_1 \cdots D_{d-5} \cdot E_1^3 \cdot ((a_2-2a_1)D_{d-1} \cdot E_2+(a_1-a_2)^2D_{d-1}^2) \\
&=D_1 \cdots D_{d-5} \cdot E_1^2
\cdot (E_3-a_1D_{d-1}) \cdot ((a_2-2a_1)D_{d-1} \cdot E_2+(a_1-a_2)^2D_{d-1}^2) \\
&=(a_1-a_2)^2 D_1 \cdots D_{d-3} \cdot E_1^2 \cdot E_3
-a_1(a_2-2a_1) D_1 \cdots D_{d-3} \cdot E_1^2 \cdot E_2 \\
& -a_1(a_1-a_2)^2 D_1 \cdots D_{d-2} \cdot E_1^2 \\
&=-(a_1-a_2)^3+a_1^2(a_2-2a_1)-a_1(a_1-a_2)^2 \leq 0.
\end{align*}
Thus, $(\mathrm{ch}_4(X) \cdot V) \leq 0$.

(2) {\it The case where $s \geq 4$}.
Let $X=X_\Sigma$ be a toric manifold of $\rho(X)=2$
such that $\G(\Sigma)=\{x_1, \ldots, x_{d-s+2}, y_1, \ldots, y_s\}$
and the primitive relations of $\Sigma$ are
\begin{equation*}
x_1+\cdots+x_{d-s+2}=a_1y_1+\cdots+a_{s-1}y_{s-1} \mbox{ and } y_1+\cdots+y_s=0,
\end{equation*}
where $a_1 \geq \cdots \geq a_{s-1} \geq 0$.
Then $X$ is the $\mathbb{P}^{s-1}$-bundle
$\mathbb{P}_{\mathbb{P}^{d-s+1}}
(\mathcal{O}(a_1) \oplus \cdots \oplus \mathcal{O}(a_{s-1}) \oplus \mathcal{O})$
over $\mathbb{P}^{d-s+1}$.
We denote by
$D_1, \ldots, D_{d-s+2}, E_1, \ldots, E_s$
the torus invariant divisors in $\Pic(X)$
associated to $x_1, \ldots, x_{d-s+2}, y_1, \ldots, y_s$, respectively.
Then we have the relations
\begin{equation*}
D_1=\cdots=D_{d-s+2} \mbox{ and } a_iD_{d-s+2}+E_i-E_s=0
\end{equation*}
for $i=1, \ldots, s-1$ in $\Pic(X)$.
Put $\tau$ as the $(d-4)$-dimensional cone in $\Sigma$
whose generators are $x_1, \ldots, x_{d-s}, y_1, \ldots, y_{s-4}$.
We denote by $V$ the $4$-dimensional torus invariant closed subvariety
of $X$ associated to $\tau$.
We show $(\mathrm{ch}_4(X) \cdot V) \leq 0$.
For every $i=1, \ldots, d-s+2$, we have
$D_i^4 \cdot D_1 \cdots D_{d-s} \cdot E_1 \cdots E_{s-4}=0$.
We calculate
\begin{align*}
&E_1^4 \cdot D_1 \cdots D_{d-s} \cdot E_1 \cdots E_{s-4} \\
&=D_1 \cdots D_{d-s} \cdot E_1^2 \cdot E_2 \cdots E_{s-4}
\cdot (E_{s-3}+(a_{s-3}-a_1)D_{d-s+2}) \\
&\cdot (E_{s-2}+(a_{s-2}-a_1)D_{d-s+2})
\cdot (E_{s-1}+(a_{s-1}-a_1)D_{d-s+2}) \\
&=D_1 \cdots D_{d-s} \cdot E_1^2 \cdot E_2 \cdots E_{s-4}
\cdot (E_{s-3} \cdot E_{s-2} \cdot E_{s-1}
+(a_{s-3}-a_1) D_{d-s+2} \cdot E_{s-2} \cdot E_{s-1} \\
&+(a_{s-2}-a_1) D_{d-s+2} \cdot E_{s-3} \cdot E_{s-1}
+(a_{s-1}-a_1) D_{d-s+2} \cdot E_{s-3} \cdot E_{s-2}) \\
&=-a_1+(a_{s-3}-a_1)+(a_{s-2}-a_1)+(a_{s-1}-a_1) \leq 0.
\end{align*}
Similarly,
\begin{equation*}
E_i^4 \cdot D_1 \cdots D_{d-s} \cdot E_1 \cdots E_{s-4}
=-a_i+(a_{s-3}-a_i)+(a_{s-2}-a_i)+(a_{s-1}-a_i) \leq 0
\end{equation*}
for every $i=2, \ldots, s-4$.
Hence it suffices to show
\begin{equation*}
(E_{s-3}^4+E_{s-2}^4+E_{s-1}^4+E_s^4) \cdot D_1 \cdots D_{d-s} \cdot E_1 \cdots E_{s-4}
\leq 0.
\end{equation*}
We calculate
\begin{align*}
&E_{s-3}^4 \cdot D_1 \cdots D_{d-s} \cdot E_1 \cdots E_{s-4} \\
&=D_1 \cdots D_{d-s} \cdot E_1 \cdots E_{s-4} \cdot E_{s-3}^2 \\
&\cdot (E_{s-2}+(a_{s-2}-a_{s-3})D_{d-s+2})
\cdot (E_{s-1}+(a_{s-1}-a_{s-3})D_{d-s+2}) \\
&=D_1 \cdots D_{d-s} \cdot E_1 \cdots E_{s-4} \cdot E_{s-3}^2 \\
&\cdot (E_{s-2} \cdot E_{s-1}
+(a_{s-2}-a_{s-3})D_{d-s+2} \cdot E_{s-1}
+(a_{s-1}-a_{s-3})D_{d-s+2} \cdot E_{s-2}) \\
&=-a_{s-3}+(a_{s-2}-a_{s-3})+(a_{s-1}-a_{s-3}) \\
&=(a_{s-3}+a_{s-2}+a_{s-1})-4a_{s-3}.
\end{align*}
Similarly,
\begin{align*}
&E_{s-2}^4 \cdot D_1 \cdots D_{d-s} \cdot E_1 \cdots E_{s-4}
=(a_{s-3}+a_{s-2}+a_{s-1})-4a_{s-2}, \\
&E_{s-1}^4 \cdot D_1 \cdots D_{d-s} \cdot E_1 \cdots E_{s-4}
=(a_{s-3}+a_{s-2}+a_{s-1})-4a_{s-1}.
\end{align*}
Furthermore, we calculate
\begin{align*}
&E_s^4 \cdot D_1 \cdots D_{d-s} \cdot E_1 \cdots E_{s-4} \\
&=D_1 \cdots D_{d-s} \cdot E_1 \cdots E_{s-4} \cdot E_s^2
\cdot (E_{s-2}+a_{s-2}D_{d-s+2}) \cdot (E_{s-1}+a_{s-1}D_{d-s+2}) \\
&=D_1 \cdots D_{d-s} \cdot E_1 \cdots E_{s-4} \cdot E_s^2
\cdot (E_{s-2} \cdot E_{s-1}+a_{s-2}D_{d-s+2} \cdot E_{s-1}+a_{s-1}D_{d-s+2} \cdot E_{s-2}) \\
&=a_{s-3}+a_{s-2}+a_{s-1}.
\end{align*}
Hence we have
\begin{align*}
&(E_{s-3}^4+E_{s-2}^4+E_{s-1}^4+E_s^4) \cdot D_1 \cdots D_{d-s} \cdot E_1 \cdots E_{s-4} \\
&=4(a_{s-3}+a_{s-2}+a_{s-1})-4a_{s-3}-4a_{s-2}-4a_{s-1}=0.
\end{align*}
Thus, $(\mathrm{ch}_4(X) \cdot V) \leq 0$.

In every case, $\mathrm{ch}_4(X)$ is not positive.
\end{proof}

%%%%%%%%%%%%%%%%%%%%%%%%%%%%%%%%%
%%%%%%%%%%%%%%%%%%%%%%%%%%%%%%

\end{document}